\documentclass{article}

\usepackage[british]{babel}
\usepackage[applemac]{inputenc}
\usepackage{amsmath}
\usepackage{mathtools}
\usepackage{amssymb}
\usepackage{amsthm}
\usepackage{mathrsfs}
\usepackage{xcolor}
\usepackage{enumitem}
\usepackage{geometry}
\usepackage{hyperref}
\usepackage[nameinlink]{cleveref}

\newtheorem{theorem}{Theorem}

\newtheorem{lemma}[theorem]{Lemma}
\newtheorem{corollary}{Corollary}

\theoremstyle{remark}

\newcounter{saveenumerate}
\makeatletter
\newcommand{\enumeratext}[1]{%
\setcounter{saveenumerate}{\value{enum\romannumeral\the\@enumdepth}}
\end{enumerate}
#1
\begin{enumerate}
\setcounter{enum\romannumeral\the\@enumdepth}{\value{saveenumerate}}%
}
\makeatother

\DeclarePairedDelimiterX{\norm}[1]{\lVert}{\rVert}{#1}

\newcommand{\N}{\ensuremath{\mathbb{N}}}

\newcommand{\R}{\ensuremath{\mathbb{R}}}

\newcommand{\tr}{\text{tr}}

\newcommand{\Gr}{\text{Gr}}

\newcommand{\bangle}[1]{\left\langle #1 \right\rangle}
\newcommand{\inprod}[2]{\bangle{#1, #2}}

\title{Equiangular subspaces in Euclidean spaces}
\author{
Igor Balla \thanks{Department of Mathematics, ETH, 8092 Zurich. igor.balla@math.ethz.ch.}
\and
Benny Sudakov \thanks{Department of Mathematics, ETH, 8092 Zurich. benjamin.sudakov@math.ethz.ch. Research supported in part by SNSF grant 200021-175573.}
}
\date{}

\begin{document}
\maketitle

\begin{abstract}
A set of lines through the origin is called \emph{equiangular} if every pair of lines defines the same angle, and the maximum size of an equiangular set of lines in $\R^n$ was studied extensively for the last 70 years. In this paper, we study analogous questions for $k$-dimensional subspaces. We discuss natural ways of defining the angle between $k$-dimensional subspaces and correspondingly study the maximum size of an equiangular set of $k$-dimensional subspaces in $\R^n$. Our bounds extend and improve a result of Blokhuis.
\end{abstract}

\section{Introduction} \label{introduction}

A set of lines passing through the origin is called \emph{equiangular} if every pair of lines makes the same angle. The question of determining the maximum size $N(n)$ of a set of equiangular lines in $\R^n$ has a long history going back 70 years. It is considered to be one of the founding problems of algebraic graph theory, see \cite{BDKS16, BY14, DGS75, GR01, LS73, vLS66} and references for more information. It is known that $N(n)$ grows quadratically with $n$. The upper bound 
\begin{equation} \label{Gerzon}
N(n) \leq \binom{n+1}{2}
\end{equation} 
was proved by Gerzon (see \cite{LS73}) and de Caen \cite{dC00} gave a (quite nontrivial) construction showing 
\begin{equation} \label{deCaen}
N(n) \geq \frac{2}{9}(n+1)^2
\end{equation} 
for all $n$ of the form $3\cdot 2^{2t-1}-1$ where $t \in \N$.

It is therefore natural and interesting to study analogous questions for $k$-dimensional subspaces. To this end, we must first understand the notion of angle between subspaces. We define the \textit{Grassmannian} $\Gr(k,n)$ to be the set of all $k$-dimensional subspaces of $\R^n$. Note that $\theta$ is the common angle between a pair of lines $U, V \in \Gr(1,n)$ if and only if
\[ \cos{\theta} = \max_{\substack{u \in U, v \in V \\ |u| = 1, |v| = 1}}{\inprod{u}{v}}.\] 
Generalizing this idea, given a pair of $k$-dimensional subspaces $U,V \in \Gr(k,n)$, we may recursively define the $k$ \emph{principal angles} $0 \leq \theta_1 \leq \ldots \leq \theta_k \leq \pi/2$ between $U$ and $V$ as follows: Choose unit vectors $u \in U, v \in V$ that maximize $\inprod{u}{v}$ and define $\theta_1 = \arccos{\inprod{u}{v}}$. Now recursively define $\theta_2, \ldots, \theta_k$ to be the principal angles between the $(k-1)$-dimensional  subspaces $U' = \{u' \in U :  u' \perp u = 0\}$ and $V' = \{v' \in V : v' \perp v = 0\}$. The geometric significance of principal angles are that they completely characterize the relative position of $U$ to $V$, in the sense that if $U', V' \in \Gr(k,n)$ have the same principal angles as $U, V$, then there exists an orthogonal matrix $Q$ such that $U' = \{Qu : u \in U\}$ and $V' = \{Qu : u \in U\}$, see \cite[Theorem 3]{W67}. 

It will be convenient for us to give another definition of principal angles that is more algebraic. Indeed, observe that a pair of lines $U, V \in \Gr(1,n)$ has common angle $\theta$ if and only if, when we choose any unit vectors $u  \in U, v \in V$, we have $(\cos{\theta})^2 = \inprod{u}{v}^2$. More generally, we associate to a subspace $U \in \Gr(k,n)$, a representative $n \times k$ matrix $\textbf{U} = (u_1, \ldots, u_k)$ where $u_1, \ldots, u_k$ is any orthonormal basis of column vectors spanning $U$. Now given a pair of subspaces $U, V \in \Gr(k,n)$  with principal angles $\theta_1, \ldots, \theta_k$, one can show that $\cos{\theta_1}, \ldots, \cos{\theta_k}$ are precisely the singular values of $\textbf{U}^{\intercal} \textbf{V}$. In other words, $(\cos{\theta_1})^2, \ldots, (\cos{\theta_k})^2$ are precisely the eigenvalues of $\textbf{V}^{\intercal} \textbf{U} \textbf{U}^{\intercal} \textbf{V}$.

Now that we understand angles between subspaces, we are ready to discuss the notion of equiangular subspaces. Note that one can consider equiangular sets of subspaces with respect to the principal angle $\theta_i$ for any fixed $1 \leq i \leq k$. More generally, for any function $d = d(\theta_1, \ldots, \theta_k)$ of the principal angles, we call a set of $k$-dimensional subspaces $H \subseteq \Gr(k,n)$ \emph{equiangular (with respect to $d$ and having common angle $\alpha$)} if $d(U,V) = \alpha$ for all $U \neq V \in H$. Thus we may define and study $N^d_{\alpha}(k, n)$, the maximum size of a set $H \subseteq \Gr(k,n)$ that is equiangular with respect to $d$ and having common angle $\alpha$, as well as $N^d(k, n) = \max_{\alpha}{N^d_{\alpha}(k,n)}$. We call a function $d : \Gr(k,n)^2 \rightarrow \R$ an \textit{angle distance} if $d(U,V) \in \{\theta_1(U,V), \ldots, \theta_k(U,V)\}$ for all $U,V \in \Gr(k,n)$. If $d$ satisfies $d(U,V) = 0$ iff $U = V$ then we call $d$ a \textit{proper} distance. 

In \cref{Angle distances} we give examples of angle distances and prove a general upper bound on $N^d_{\alpha}(k,n)$ for any angle distance $d$ and $\alpha > 0$, in particular improving and extending a result of Blokhuis \cite{B93} who studied the case $d = \theta_1$ and $k = 2$. Based on equiangular lines, we also give a lower bound construction of $k$-dimensional subspaces that are equiangular for any proper angle distance. We therefore conclude that for $k$ fixed and any proper angle distance $d$, $N^d(k,n) = \Theta(n^{2k})$ as $n \rightarrow \infty$. In \cref{Other distances}, we discuss $N^d(k,n)$ for some other well-studied distances $d$. In \cref{Concluding remarks}, we conclude by stating some open problems, in particular discussing another generalization of equiangular lines known as \emph{equi-isoclinic} subspaces.

\section{Angle distances} \label{Angle distances}

When trying to define the angle between two subspaces $U, V \in \Gr(k,n)$, one natural idea is to just take the minimum angle between any pair of vectors $u \in U, v \in V$. Since minimizing $\arccos{\inprod{u}{v}}$ is equivalent to maximizing $\inprod{u}{v}$, this idea gives exactly the first principal angle $\theta_1 = \theta_1(U,V)$. This angle distance was first considered by Dixmier \cite{D49}. In \cite{B93}, Blokhuis considered equiangular planes with respect to $\theta_1$ and proved that
\begin{equation} \label{Blokhuis}
N^{\theta_1}_{\alpha}(2, n) \leq \binom{2n + 3}{4}
\end{equation}
provided that the common angle $\alpha > 0$. This condition is necessary, since $\theta_1(U,V) = 0$ iff $U$ and $V$ share a nontrivial subspace, and so we could take infinitely many planes all sharing a fixed line, showing that $N^{\theta_1}_0(2,n) = \infty$. This is a troublesome property of $\theta_1$, because it shows that $\theta_1$ is not a proper distance and also that $\theta_1$ does not appeal to elementary geometric intuition. Indeed, consider a pair of planes $U, V$ in $\R^3$. They will always share a line and hence will have $\theta_1(U,V) = 0$. However, one would intuitively ascribe the angle between them to be $\theta_2(U,V)$. 

In view of this, it makes sense to define the minimum non-zero angle $\theta_F(U,V)  = \min \{\theta_i(U,V) : \theta_i(U,V) > 0 \}$. $\theta_F$ was first considered by Friedrichs \cite{F37} and it is a proper angle distance. Deutsch \cite{D95} gives applications of $\theta_1$ and $\theta_F$ to the rate of convergence of the method of cyclic projections, existence and uniqueness of abstract splines, and the product of operators with closed range.

Another proper angle distance is the maximum angle $\theta_k$, first considered by Krein, Krasnoselski, and Milman \cite{KKM48}. It was used by Asimov \cite{A85} for his ``Grand Tour,'' a method for visualizing high dimensional data by projecting to various two-dimensional subspaces and showing these projections sequentially to a human. $\theta_k$ was also considered by Conway, Hardin, and Sloane \cite{CHS96} in their paper on packing subspaces in Grassmannians.

For any angle distance $d$ and $\alpha > 0$, we give an upper bound on $N^d_{\alpha}(k, n)$ on the order of $n^{2k}$, extending Gerzon's bound in \cref{Gerzon}. In the case $d = \theta_1$ and $k = 2$, this improves Blokhuis' bound in \cref{Blokhuis}. The proof is based on the polynomial method, which was also the main tool in \cite{B93}.

\begin{theorem} \label{t_upper}
Let $k,n \in \N$ with $k \leq n$, let $d$ be an angle distance on $\Gr(k,n)$ and let $\alpha > 0$. Then
\[ N^d_{\alpha}(k, n) \leq \binom{\binom{n+1}{2} + k - 1}{k}. \]
\end{theorem}
\begin{proof}
Let $\{ U_1, \ldots, U_m \} \subseteq \Gr(k,n)$ be a set of subspaces such that $d(U_i, U_j) = \alpha$ for all $i \neq j$ and for each $U_i$, let $\textbf{U}_i = (u_1, \ldots, u_k)$ be a representative $n \times k$ matrix where $u_1, \ldots, u_k$ is any orthonormal basis of column vectors spanning $U_i$. Observe that for any $i \neq j$, since $\alpha = d(U_i, U_j)$ is a principal angle between $U_i$ and $U_j$, we have, as per the discussion in \cref{introduction}, that $(\cos{\alpha})^2$ is an eigenvalue of $\textbf{U}_i^{\intercal} \textbf{U}_j \textbf{U}_j^{\intercal} \textbf{U}_i$. Thus if we define $\lambda = (\cos{\alpha})^2$ then we have $\det \left( \textbf{U}_i^{\intercal}  \textbf{U}_j \textbf{U}_j^{\intercal} \textbf{U}_i - \lambda I_k \right) = 0$, where $I_k$ is the $k \times k$ identity matrix.

Now let $\mathscr{S} = \{ X \in \R^{n \times n} : X^{\intercal} = X \}$ be the set of all symmetric $n \times n$ matrices and define functions $f_1, \ldots, f_m : \mathscr{S} \rightarrow \R$ by
\[ f_i(X) = \det \left( \textbf{U}_i^{\intercal} X \textbf{U}_i - \frac{\lambda \tr(X)}{k} I_k \right). \]
Since $\tr(\textbf{U}_j \textbf{U}_j^{\intercal}) = \tr(\textbf{U}_j^{\intercal} \textbf{U}_j) = \tr(I) = k$, we conclude that
\[ f_i(\textbf{U}_j \textbf{U}_j^{\intercal}) = 
   \begin{cases} 
      (1 - \lambda)^k & \text{ if } i = j \\
      0 & \text{ if } i \neq j.
   \end{cases}
\]
Moreover, note that $\lambda \neq 1$ since $\alpha \neq 0$. It therefore follows that $f_1, \ldots, f_m$ are linearly independent. Indeed, if $\sum_{i=1}^{m}{c_i f_i} = 0$ for some $c_1, \ldots, c_m \in \R$, then for all $j$ we have $0 = \sum_{i=1}^{m}{c_i f_i(\textbf{U}_j \textbf{U}_j^{\intercal})} = c_j (1-\lambda)^k$, which implies $c_j = 0$.

Thus it suffices to show that $f_1, \ldots, f_m$ live in a space of dimension $\binom{\binom{n+1}{2} + k - 1}{k}$. To that end, recall that a multivariable polynomial $f : \R^t \rightarrow \R$ is called homogeneous of degree $k$ if it is a linear combination of monomials of degree $k$, and that the linear space of such polynomials has dimension $\binom{t + k - 1}{k}$. For any $X \in \mathscr{S}$, we let $X_{a,b}$ denote the entry in position $a,b$ of the matrix $X$, so that $\mathscr{S}$ may be parametrized by the $\binom{n+1}{2}$ variables $\{ X_{a,b} : 1 \leq a \leq b \leq n \}$ living on or above the diagonal and hence we may think of the functions $f_i$ as polynomials in these variables. Now observe that for any $i$ and $X \in \mathscr{S}$, every entry of the $k \times k$ matrix $\textbf{U}_i^{\intercal} X \textbf{U}_i - \frac{\lambda \tr(X)}{k} I_k$ is a homogeneous polynomial of degree 1 in the variables $\{ X_{a,b} : 1 \leq a \leq b \leq n \}$. It follows from the definition of the determinant that $f_i(X)$ is a homogeneous polynomial of degree $k$ in these variables. Since there are $\binom{n+1}{2}$ such variables, the space of all homogeneous polynomials of degree $k$ in these variables has dimension $\binom{\binom{n+1}{2} + k - 1}{k}$, completing the proof.
\end{proof}

To obtain lower bounds for this problem, it is natural to start with a construction of many equiangular lines and then try to combine them to make $k$-dimensional subspaces. Recall that $N(n)$ is the maximum size of a set of equiangular lines in $\R^n$. In the following, we make use of the Frobenius inner product $\inprod{A}{B} = \tr(A^{\intercal} B)$ for $n \times n$ real-valued matrices $A, B$.

\begin{theorem} \label{t_lower}
For any $k, n \in \N$ with $k \leq n$, there exists a set $H \subseteq \Gr(k, kn)$ with $|H| = N(n)^k$ and $\alpha \in (0, \pi/2)$ such that for all $U,V \in H$, the principal angles between $U$ and $V$ all lie in the set $\{ 0, \alpha \}$.
\end{theorem}
\begin{proof}
Let $L \subseteq \Gr(1,n)$ be an equiangular set of lines with $|L| = N(n)$, and let $\alpha \in (0, \pi/2)$ be the common angle of any pair of lines in $L$. Now let $C$ be the set of vectors obtained by choosing a unit vector along each line in $L$, and observe that $\inprod{u}{v}^2 = (\cos{\alpha})^2$ for all $u \neq v \in C$.

Now let $e_1, \ldots e_k$ be the standard basis in $\R^k$, and observe that for all $u, v \in C$, we have
\[ \inprod{e_i u^{\intercal}}{e_j v^{\intercal}} = \tr(u e_i^{\intercal} e_j v^{\intercal}) = (e_i^{\intercal} e_j) (u^{\intercal} v) =
   \begin{cases} 
      \inprod{u}{v} & \text{ if } i = j \\
      0 & \text{ if } i \neq j.
   \end{cases}
\]
Now observe that for any $i$ and $u \in C$, $e_i u^{\intercal}$ can be viewed as a vector in $\R^{kn}$ and thus if we let $u_1, \ldots, u_k \in C$, then $e_1 u_1^{\intercal}, \ldots, e_k u_k^{\intercal}$ can be viewed as orthonormal vectors in $\R^{kn}$ and hence define a subspace $W_{u_1, \ldots, u_k}$ in $\Gr(k, kn)$. Furthermore, for all $u_1, \ldots, u_k, v_1, \ldots, v_k \in C$, if we let $\textbf{U}$ be the $kn \times k$ matrix with column vectors $e_1 u_1^{\intercal}, \ldots, e_k u_k^{\intercal}$ and let $\textbf{V}$ be the $kn \times k$ matrix with column vectors $e_1 v_1^{\intercal}, \ldots, e_k v_k^{\intercal}$, then $\textbf{U}$ is a representative matrix for $W_{u_1, \ldots, u_k}$ and $\textbf{V}$ is a representative matrix for $W_{v_1, \ldots, v_k}$. Now we compute that
\[ 
 (\textbf{U}^{\intercal} \textbf{V})_{i,j} = \inprod{e_i u_i^{\intercal}}{e_j v_j^{\intercal}} = 
   \begin{cases} 
      \inprod{u_i}{v_j} & \text{ if } i = j \\
      0 & \text{ if } i \neq j,
   \end{cases}
\]
and hence
\[ 
 (\textbf{V}^{\intercal} \textbf{U} \textbf{U}^{\intercal} \textbf{V})_{i,j} = 
   \begin{cases} 
      \inprod{u_i}{v_j}^2 & \text{ if } i = j \\
      0 & \text{ if } i \neq j.
   \end{cases}
\]
Thus the eigenvalues of $\textbf{V}^{\intercal} \textbf{U} \textbf{U}^{\intercal} \textbf{V}$ lie in the set $\{1, \cos(\alpha)^2\}$ and so the principal angles between $W_{u_1, \ldots, u_k}$ and $W_{v_1, \ldots, v_k}$ lie in the set $\{0, \alpha\}$. Letting $H = \{ W_{u_1, \ldots, u_k} : u_1, \ldots, u_k \in C\}$ and observing that $|H| = N(n)^k$ completes the proof.
\end{proof}

Next we show that the construction above is equiangular for any proper angle distance $d$, and hence obtain the following corollary.

\begin{corollary} \label{c1}
Let $d$ be a proper angle distance and let $k \in \N$ be fixed. Then 
\[ N^{d}(k, n) = \Theta(n^{2k}) \text{ as } n \rightarrow \infty. \]
\end{corollary}
\begin{proof}
\Cref{t_upper} immediately gives the upper bound $N^d(k, n) \leq O(n^{2k})$. For the lower bound, let $\alpha \in (0, \pi/2)$ and $H \subseteq \Gr(k, kn)$ be given by \Cref{t_lower}. Observe that for all $U \neq V \in H$, the principal angles between $U$ and $V$ cannot all be $0$, and thus $\theta_F(U,V) = \theta_k(U,V) = \alpha$. Moreover, observe that since $d$ is a proper angle distance, we have $\theta_F \leq d \leq \theta_k$. Thus $d(U,V) = \alpha$ for all $U \neq V \in H$. De Caen's bound \cref{deCaen} implies that $N(n) \geq \Omega(n^2)$ and so we obtain
\[ N^d(k,kn) \geq |H| = N(n)^k \geq \Omega(n^{2k}).\]
Thus we conclude $N^d(k,n) \geq \Omega(n^{2k})$.
\end{proof}

\section{Other distances} \label{Other distances}

Besides angle distances, there are several other natural distance functions that are considered in geometry, statistics, and applied problems, see e.g. \cite{EAS98}. Let $U,V \in \Gr(k,n)$ be $k$-dimensional subspaces of $\R^n$ with principal angles $\theta_1, \ldots, \theta_k$. If one considers the Grassmanian $\Gr(k,n)$ as a manifold, one may compute (see \cite[Theorem 8]{W67}) that the geodesic distance is 
\[ d_G(U,V) = \sqrt{\theta_1^2 + \ldots + \theta_k^2}.\] 
In the context of packing subspaces, Conway, Hardin, and Sloane \cite{CHS96} consider the geodesic distance, the maximum principal angle, as well as the chordal distance defined by 
\[ d_C(U,V) = \sqrt{(\sin{\theta_1})^2 + \ldots + (\sin{\theta_k})^2} = \sqrt{k - \tr(\textbf{V}^{\intercal} \textbf{U} \textbf{U}^{\intercal} \textbf{V})}.\] 
Also in the context of packing subspaces, Dhillon, Heath, Strohmer, and Tropp \cite{DHST08} consider the first principal angle (spectral distance), as well as the Fubini-Study distance defined by
\[ d_{FS}(U,V) = \arccos\left(\prod_{i=1}^{k}{\cos{\theta_i}}\right) = \arccos{\left | \det{\textbf{U}^{\intercal} \textbf{V}} \right |}.\]

For a subspace $U \in \Gr(k,n)$, we define the orthogonal complement $U^{\perp} = \{v \in \R^n : v \perp u \text{ for all } u \in U\}$. The following lemma shows us that the nonzero principal angles between subspaces are the same as the nonzero principal angles between their orthogonal complements.

\begin{lemma} \label{l_complement}
For any $U,V \in \Gr(k,n)$, the nonzero principal angles between $U^{\perp}$ and $V^{\perp}$ are the same as the nonzero principal angles between $U$ and $V$. 
\end{lemma}
\begin{proof}
Observe that $\textbf{U} \textbf{U}^{\intercal}$ is an orthogonal projection onto $U$ and $\textbf{U}^{\perp} (\textbf{U}^{\perp})^{\intercal}$ is an orthogonal projection onto $U^{\perp}$, so that $\textbf{U} \textbf{U}^{\intercal} + \textbf{U}^{\perp} (\textbf{U}^{\perp})^{\intercal} = I_n$ where $I_n$ is the $n \times n$ identity matrix. Thus 
\[ 
\textbf{U}^{\intercal} \textbf{V} \textbf{V}^{\intercal} \textbf{U}
= \textbf{U}^{\intercal} (I_n - \textbf{V}^{\perp} (\textbf{V}^{\perp})^{\intercal} ) \textbf{U}
= I_{k} - \textbf{U}^{\intercal} \textbf{V}^{\perp} (\textbf{V}^{\perp})^{\intercal}  \textbf{U}
\]
and
\[ 
(\textbf{V}^{\perp})^{\intercal} \textbf{U}^{\perp} (\textbf{U}^{\perp})^{\intercal} \textbf{V}^{\perp}
= (\textbf{V}^{\perp})^{\intercal} (I_n - \textbf{U} \textbf{U}^{\intercal} ) \textbf{V}^{\perp}
= I_{n-k} - (\textbf{V}^{\perp})^{\intercal} \textbf{U} \textbf{U}^{\intercal} \textbf{V}^{\perp}.
\]
Since it is well known that for any $A, B$ the matrices $AB$ and $BA$ have the same nonzero eigenvalues with the same multiplicity, we have that $\textbf{U}^{\intercal} \textbf{V}^{\perp} (\textbf{V}^{\perp})^{\intercal}  \textbf{U}$ has the same nonzero eigenvalues as $(\textbf{V}^{\perp})^{\intercal} \textbf{U} \textbf{U}^{\intercal} \textbf{V}^{\perp}$, and therefore $\textbf{U}^{\intercal} \textbf{V} \textbf{V}^{\intercal} \textbf{U}$ has the same eigenvalues as $(\textbf{V}^{\perp})^{\intercal} \textbf{U}^{\perp} (\textbf{U}^{\perp})^{\intercal} \textbf{V}^{\perp}$ except for eigenvalues of 1. Hence the principal angles between $U$ and $V$ are the same as the principal angles between $U^{\perp}$ and $V^{\perp}$, except for angles of $0$.
\end{proof}

Now let $d$ be one of the proper distances discussed in this paper, and observe that principal angles of 0 don't affect $d$. Thus using \Cref{l_complement}, we have that $d(U^{\perp}, V^{\perp}) = d(U,V)$ for all $U,V \in \Gr(k,n)$. We therefore conclude that $U_1, \ldots, U_m \in \Gr(k,n)$ are equiangular with respect to $d$ iff $U_1^{\perp}, \ldots, U_m^{\perp} \in \Gr(n-k,n)$ are equiangular with respect to $d$, and hence that
\[ N^d(k,n) = N^d(n-k,n).\]
Thus, for the purposes of studying $N^d(k,n)$, it will suffice for us to consider the case $k \leq n/2$.

Conway, Hardin, and Sloane \cite{CHS96} give some reasons why they consider the chordal distance $d_C$ to be the best definition for packings, in particular observing that the Grassmanian $\Gr(k,n)$ with the chordal distance can be isometrically embedded onto a sphere in $\R^D$ for $D = \binom{n+1}{2} - 1$, by mapping a subspace $U$ to the projection matrix $\textbf{U}\textbf{U}^{\intercal}$ and using the Frobenius inner product $\tr(A^{\intercal} B)$. Since an equidistant set (simplex) in $\R^D$ has size at most $D+1$, they conclude that 
\[ N^{d_C}(k,n) \leq \binom{n+1}{2},\] 
generalizing Gerzon's bound \cref{Gerzon}. For a lower bound, given a set of $m$ $k$-dimensional subspaces $U_1, \ldots, U_m \in \Gr(k,n)$ equiangular with respect to $d_C$, observe that by adding a new dimension and defining $U_i' = \text{span}(U_i, e_{n+1})$, we obtain a set of $m$ $(k+1)$-dimensional subspaces in $\Gr(k+1, n+1)$ which is equiangular with respect to $d_C$. Thus $N^{d_C}(k+1,n+1) \geq N^{d_C}(k,n)$ for all $k \leq n$ and so, using the assumption $k \leq n/2$ together with \cref{deCaen}, we obtain
\[ N^{d_C}(k,n) \geq N^{d_C}(1,n-k+1) = N(n-k+1) = \Omega(n^2).\]
Additionally, for a prime $p$ such that a Hadamard matrix of order $(p+1)/2$ exists, Calderbank, Hardin, Rains, Shor, and Sloane \cite{CHRSS99} give a construction of $\binom{p+1}{2}$ subspaces of dimension $(p-1)/2$ in $\R^p$ which are equiangular with respect to $d_C$, so that 
\[ N^{d_C}((p-1)/2, p) = \binom{p+1}{2}.\]

For the Fubini-Study distance $d_{FS}$, we will need some definitions from multilinear algebra, see e.g. \cite{Y92} for reference. Let $u \wedge v$ denote the wedge product between $u,v \in \R^n$. Let $\bigwedge^k(\R^n) = \{u_1 \wedge \ldots \wedge u_k : u_1, \ldots, u_k \in \R^n\} $ denote the $k$th exterior power of $\R^n$ and note that $\dim{\bigwedge^k(\R^n)} = \binom{n}{k}$. We shall use the Pl\"{u}cker embedding of $\Gr(k,n)$ into the projective space of lines over $\bigwedge^k(\R^n)$, defined as follows. Given a subspace $U \in \Gr(k,n)$ with $u_1, \ldots, u_k$ being an orthonormal basis of column vectors of $\textbf{U}$, we define $\phi(U) = u_1 \wedge \ldots \wedge u_k \in \bigwedge^k(\R^n)$. One can compute that $\langle \phi(U), \phi(V) \rangle = \det(\textbf{U}^{\intercal} \textbf{V})$ defines an inner product between $\phi(U)$ and $\phi(V)$. Therefore, given a set of subspaces $U_1, \ldots, U_m \in \Gr(k,n)$ equiangular with respect to $d_{FS}$, we have that $\phi(U_1), \ldots, \phi(U_m)$ are a set of vectors such that if we take a line along each vector, we obtain a set of equiangular lines in $\bigwedge^k(\R^n)$. Thus using \cref{Gerzon}, we conclude
\[ N^{d_{FS}}(k,n) \leq \binom{\binom{n}{k}+1}{2}. \]
Actually, the Pl\"{u}cker embedding gives an embedding into an algebraic variety over $\bigwedge^k(\R^n)$  defined by the so-called Pl\"{u}cker relations, and so it conceivable that this can be used to obtain a better upper bound. If a matching lower bound construction exists, finding it seems difficult since it would, in particular, yield a new construction of $\Omega(N^2)$ equiangular lines in $\R^N$, for $N = \binom{n}{k}$.

We do not know anything about equiangular subspaces for the geodesic distance $d_{G}$, as well as other distances which cannot be written in terms of polynomial expressions of $\cos{\theta_1}, \ldots, \cos{\theta_k}$. This is not surprising, since all of the above upper bounds are essentially proven via the polynomial method. It would, therefore, be interesting to find other methods for proving such upper bounds.

\section{Concluding remarks} \label{Concluding remarks}
 
In \cref{Angle distances} we give an upper bound on $N_{\alpha}^d(k,n)$ of the order $n^{2k}$ for any angle distance $d$ and $\alpha > 0$, but are only able to give a corresponding lower bound when $d$ is a proper angle distance. It would therefore be interesting to give lower bound constructions (with common angle $\alpha > 0$) on the order of $n^{2k}$ for angle distances that are not proper, in particular for the minimum angle $\theta_1$. Moreover, if $n \gg k \rightarrow \infty$ then even for proper angle distances $d$, \Cref{c1} still leaves open the correct asymptotic dependence of $N^d(k,n)$ on $k$. 

In section 3, we remark that the polynomial method does not seem to work for distances such as the geodesic distance $d_G$, and so it would be interesting to find new methods which give upper bounds for such cases. It would also be interesting to obtain lower bound constructions for the Fubini-study distance $d_{FS}$, and establish the correct order of magnitude for $N^{d_{FS}}(k,n)$.

Another approach to generalizing equiangular lines is, given a set $H \subseteq \Gr(k, n)$, to require that $H$ is equiangular with respect to $\theta_i$ for all $1 \leq i \leq k$. If we further require that $\theta_1 = \ldots = \theta_k$, we arrive at the notion of \textit{equi-isoclinic} subspaces. Equivalently, a family of subspaces $H \subseteq \Gr(k,n)$ is equi-isoclinic if there exists $\lambda \in [0,1)$ such that $V^{\intercal} U U^{\intercal} V = \lambda I$ for all $U \neq V \in H$. Lemmens and Seidel \cite{LSE73} defined and studied $v(k,n)$, the maximum number of $k$-dimensional equi-isoclinic subspaces in $\R^n$. They gave a construction based on equiangular lines showing that $v(k, kn) \geq v(1,n)$ and generalized Gerzon's bound in \cref{Gerzon}, obtaining $v(k,n) \leq \binom{n+1}{2} - \binom{k+1}{2} + 1$. Note that for $n \gg k \rightarrow \infty$, these bounds together with the fact that $v(1,n) = N(n) \geq \Omega(n^2)$ show that
\[ \Omega \left( \frac{n^2}{k^2} \right) \leq v(k,n) \leq O(n^2).\]
It would be interesting to close this gap and determine the correct asymptotic dependence of $v(k,n)$ on $k$.

\end{document}